\documentclass{amsart}

\usepackage{amsmath}%
\usepackage{amsfonts}%
\usepackage{amssymb}%
\usepackage[hyperfootnotes=false, colorlinks]{hyperref}
\usepackage{fncylab}

\DeclareMathOperator{\decomp}{\sqcup}
\newcommand{\NN}{\mathbb N}
\newcommand{\set}[1]{\left\{\,#1\,\right\}}  
\newcommand{\with}{\ \vrule\ }
\DeclareMathOperator{\dcup}{\dot{\cup}}
\newcommand{\defa}{:=}
\newcommand{\comp}[1]{\overline{#1}}

\newcommand{\cc}[1]{\binom{[n]}{2} \setminus #1}

\newcommand{\wo}[1]{\omega_{0,{#1}}}
\newcommand{\rev}[1]{\check{#1}}
\newcommand{\id}[1]{id_{#1}}
\newcommand{\Sn}{S_n}
\newcommand{\T}[1][\pi]{T(#1)}
\newcommand{\G}[1][\pi]{G(#1)}

\makeatletter
\newcommand{\numberlike}[2]{%
	\@namedef{c@#1}{\@nameuse{c@#2}}%
}
\makeatother
\newcommand{\mynewtheorem}[2]{
	\newtheorem{#1}{#2}[section]
	\labelformat{#1}{#2~##1}
}

\labelformat{section}{Section~#1}

\theoremstyle{plain}
\mynewtheorem{lemma}{Lemma}
\mynewtheorem{theorem}{Theorem}
\numberlike{theorem}{lemma}
\mynewtheorem{corollary}{Corollary}
\numberlike{corollary}{lemma}
\mynewtheorem{proposition}{Proposition}
\numberlike{proposition}{lemma}
\mynewtheorem{algorithm}{Construction}
\numberlike{algorithm}{lemma}

\theoremstyle{definition}
\mynewtheorem{definition}{Definition}
\numberlike{definition}{lemma}
\mynewtheorem{problem}{Problem}
\numberlike{problem}{lemma}

\begin{document}
\title{Decomposing Sets of Inversions}
\author{Lukas Katth\"an}
\address{Fachbereich Mathematik und Informatik, Philipps-Universit\"at Marburg 
}%
\email{katthaen@mathematik.uni-marburg.de}%
\thanks{This work was partially supported by the DAAD}

\date{\today}
\subjclass[2010]{Primary 05E40, 52B12; Secondary 05A05} %
\keywords{Inversion sets, Permutation graphs, Simple Permutations, Linear ordering polytope}%

\begin{abstract}
In this paper we consider the question how the set of inversions of a permutation $\pi \in S_n$ can be partitioned into two subsets, which are themselves inversion sets of permutations in $S_n$.
Our method is to study the modular decomposition of the inversion graph of $\pi$.
A correspondence to the substitution decomposition of $\pi$ is also given.
Moreover, we consider the special case of multiplicative decompositions.
\end{abstract}
\maketitle
%
%
\section{Introduction}\label{sec:intro}
For a permutation $\pi \in \Sn$ denote its \emph{inversion set} by
\[ \T \defa \set{ \set{i,j} \in \NN^2 \with 1 \leq i < j \leq n, \pi(i) > \pi(j) } \,.\]
In this paper, we address the following problem:
\begin{problem}\label{problem}
For a given permutation $\pi \in \Sn$, give a description of all $\tau_1, \tau_2 \in \Sn$ such that
\begin{equation}\label{eq:zerlegung}\begin{split}
\T[\tau_1] &\cup \T[\tau_2] = \T \\
\T[\tau_1] &\cap \T[\tau_2] = \emptyset \\
\tau_1, \tau_2 &\neq \id{n} \;.
\end{split}
\end{equation}
\end{problem}
In other words, we want to find all ways to distribute the inversions of $\pi$ into two disjoint sets, such that each is itself the inversion set of a permutation.
The motivation behind this problem is as follows. In \cite{welker}, toric statistical ranking models are considered. One of these model is the \emph{inversion model}, which is also known as \emph{Babington-Smith Model} in the statistics literature, see \cite{marden}. The toric ideal $I_{BS}$ associated to this model is the kernel of the map
\begin{align*}
k[X_{\pi} \with \pi \in \Sn] &\rightarrow k[X_{ij} \with 1\leq i < j \leq n] \\
X_{\pi} &\mapsto \prod_{\set{i,j}\in \T} X_{ij} 
\end{align*}
It follows from general theory that $I_{BS}$ is generated by differences of monomials \cite[Lemma 4.1]{sturmfels}.
By construction, a difference of monomials $\prod_i X_{\pi_i} - \prod_i X_{\tau_i}$ is contained in $I_{BS}$ if and only if  $\bigcup_i \T[\pi_i] = \bigcup_i \T[\tau_i]$ as multisets. Thus the generators of the ideal $I_{BS}$ encode the relations among the inversion sets of permutations.
Therefore, a set of generators for this ideal not only provides algebraic information for the Babington-Smith Model but also encodes fundamental information about the combinatorics of permutations.
However, $I_{BS}$ turns out to be a rather large and complex object, for example the authors of \cite{welker} found computationally that for $n=6$ there are as many as $130377$ quadratic generators and there are also generators of higher degree.
Therefore, as a first step in understanding this object, we study its quadratic generators for all $n$.
The ideal $I_{BS}$ is invariant under the right action of the $\Sn$, so if $m := X_{\pi_1} X_{\pi_2} - X_{\tau_1} X_{\tau_2} \in I_{BS}$, then also $m \pi_1^{-1} = X_{\id{n}} X_{\pi_2\pi_1^{-1}} - X_{\tau_1\pi_1^{-1}} X_{\tau_2\pi_1^{-1}} \in I_{BS}$.
Therefore we can restrict our attention to binomials of the form
\[ X_{\id{n}} X_{\pi} - X_{\tau_1} X_{\tau_2} \,.\]
From our discussion, the following observation is immediate:
\begin{proposition}
A binomial $X_{\id{n}} X_{\pi} - X_{\tau_1} X_{\tau_2}$ lies in $I_{BS}$ if and only if $\pi,\tau_1$ and $\tau_2$ satisfy \eqref{eq:zerlegung}.
\end{proposition}
Thus \ref{problem} is equivalent to the problem of describing the quadratic generators of $I_{BS}$.
In the recent preprint \cite{wehlau}, the following closely related question is considered:
Let $\wo{n} \in \Sn$ denote the permutation of maximal length (i.e. the one mapping $i \mapsto n+1-i$).
\begin{problem}
Give a description of all sets $\set{\tau_1,\ldots,\tau_l}\subset\Sn$ such that $\T[\wo{n}] = \bigcup_i \T[\tau_i]$ and $\T[\tau_i] \cap \T[\tau_j] = \emptyset$ for $i \neq j$.
\end{problem}
The motivation and the methods employed by the authors of \cite{wehlau} are different from ours, but some intermediate results of this paper were also found independently there. In particular, \ref{lemma:deg3} and part of \ref{cor:crit} resemble Proposition 2.2 and Proposition 3.14 in \cite{wehlau}.

Another perspective on a toric model is via its \emph{model polytope}.
The model polytope associated to the inversion model is the \emph{linear ordering polytope}\cite{welker}, which is a well-studied object in combinatorial optimization, see \cite[Chapter 6]{marti}.
In \cite{young} the following question is addressed:
\begin{problem}
Which permutations $\pi \in S_n$ are neighbours of the identity permutation in the graph of the linear ordering polytope?\footnote{The linear ordering polytope is called the `permutation polytope' in \cite{young}.}
\end{problem}
In \cite{young}, a characterization of these permutations is obtained,
but as we show after \ref{cor:crit} there is a gap in the proof.
Nevertheless, the result from \cite{young} is correct and we extend the result and provide a proof in \ref{cor:crit}.
It turns out that a permutation has a decomposition as in \eqref{eq:zerlegung} if and only if it is not a neighbour of the identity permutation in the graph of the linear ordering polytope.
However, in the present paper we are interested in a description of all possible decompositions of type \eqref{eq:zerlegung}.

This paper is divided into four sections and an appendix.
In \ref{sec:prelim} we review the concept of modular decomposition for graphs, the characterisation of inversion sets of permutations and we discuss blocks of permutations.
In \ref{sec:main}, we prove our main results.
In \ref{thm:klass}, we give an answer to \ref{problem} in terms of the modular decomposition of the inversion graph of $\pi$. 
Moreover, we consider a modification of \eqref{eq:zerlegung}, where we impose the further restriction that $\pi = \tau_1 \tau_2$.
We show in \ref{thm:mult} that if $\pi$ admits a solution of \eqref{eq:zerlegung}, then it also admits a solution satisfying $\pi = \tau_1 \tau_2$.
Since \ref{problem} is formulated without referring to graphs, in \ref{thm:invdecomp} we give a reformulation of \ref{thm:klass} which avoids notions from graph theory.
In \ref{sec:further}, we show that the problem of decomposing an inversion set into three or more inversion sets can be reduced to \eqref{eq:zerlegung}. Moreover, we show that permutations of sufficiently high length always admit a solution of \eqref{eq:zerlegung}.
In the appendix we prove a result connecting the blocks of a permutation with the modules of its inversion graph.
The result from the appendix seems rather natural to us, but since we were not able to find it in the literature, we include a proof.

\section{Preliminaries}\label{sec:prelim}
\subsection{Notation}
Let us first fix some notation. 
We denote a graph $G$ on a vertex set $V$ with edge set $E \subset V \times V$ by $G=(V,E)$. All our graphs are undirected and simple. For two vertices $v,w$, let $vw$ denote the (undirected) edge between $v$ and $w$. We say $v$ and $w$ are connected in $G$ if $vw \in E$ and we write $vw \in G$ by abuse of notation.

For a natural number $n \in \NN$, we write $[n]$ for the set $\set{1,\dotsc,n}$. For a finite set $S$, we write $\binom{S}{2}$ for the set of subsets of $S$ containing exactly $2$ elements.
For $\pi \in\Sn$ we denote by $\T$ the \emph{inversion set} 
\[\set{\set{i,j}\in \binom{[n]}{2} \with i<j, \pi(i)>\pi(j)} \;.\]
This set can be considered as the edge set of an undirected graph $\G = ([n], \T)$, the \emph{inversion graph} of $\pi$.
We consider this graph without the natural order on its vertices, therefore in general $\G$ does not uniquely determine $\pi$.
The graphs arising this way are called \emph{permutation graphs}, see \cite{brandt}.
By another abuse of notation, we write $ij \in \T$ (resp. $ij \in \G$) if $\set{i,j}$ is an inversion of $\pi$.
For two subsets $A,B \subset [n]$, we write $A < B$ if $a < b$ for every $a\in A, b \in B$.
\subsection{Modular decomposition of graphs}
In this subsection we review the modular composition for graphs, see \cite[Chapter 1.5]{brandt} for a reference.
Let $G = (V,E)$ be a graph.
\begin{definition}[\cite{brandt}]
\begin{enumerate}	
	\item A set $M \subset V$ is called a \emph{module} of $G$ if for $m_1,m_2 \in M$ and $v \in V\setminus M$ it holds that $vm_1 \in G$ if and only if $vm_2 \in G$.
	\item A module $M$ is called \emph{strong} if for every other module $N$ either $M\cap N = \emptyset$, $M\subset N$ or $N \subset M$ holds.
\end{enumerate}
\end{definition}
In \cite[p. 14]{brandt} it is shown that for every module there is a unique minimal strong module containing it.
A graph is called \emph{prime} if $V$ and its vertices are its only modules.
We denote by $\comp{G}$ the complementary graph $\comp{G} = (V, \binom{V}{2}\setminus E)$ of $G$. For a subset $U \subset V$, we denote by $G_U$ the induced subgraph of $G$ on $U$.
\begin{theorem}[Theorem 1.5.1, \cite{brandt}] \label{thm:moddecomp}
Let $G=(V,E)$ be a graph with at least two vertices. Then the maximal strong submodules (\emph{m.s.s.}) of $G$ form a partition of $V$ and exactly one of the following conditions hold:
\begin{description}
	\item[Parallel case] $G$ is not connected. Then its m.s.s. are its connected components.
	\item[Serial case] $\comp{G}$ is not connected. Then the m.s.s. of $G$ are the connected components of $\comp{G}$.
	\item[Prime case] Both $G$ and $\comp{G}$ are connected. Then there is a subset $U\subset V$ such that
		\begin{enumerate}
			\item $\#U > 3$,
			\item $G_U$ is a maximal prime subgraph of $G$,
			\item and every m.s.s. $M$ of $G$ has $\# M\cap U = 1$.
		\end{enumerate}
\end{description}
\end{theorem}
We call a module $M$ of $G$ \emph{parallel}, \emph{serial} or \emph{prime} corresponding to which condition of above theorem is satisfied by $G_M$.
As a convention, we consider single vertices as parallel modules.
By the following lemma, we do not need to distinguish between modules of $G$ contained in a module $M$ and modules of $G_M$.
\begin{lemma}\label{prop:untermod}
	Let $G$ be a graph, $M$ a module of $G$ and $U \subset M$ a subset. Then $U$ is a module of $G$ if and only if it is a module of $G_M$.
	Moreover, $U$ is a strong module of $G$ if and only if it is strong as a module of $G_M$.
\end{lemma}
\begin{proof}
	The first statement is immediate from the definitions.
	For the second statement, first assume that $U$ is not strong as a module in $G_M$.
	We say that a module $N$ overlaps $U$ if $N \cap U \neq \emptyset$, $N \nsubseteq U$ and $U \nsubseteq N$ holds. 
	So by our assumption, there is a module $N \subset M$ of $G_M$ overlapping $U$. But $N$ is also a module of $G$, hence $U$ is not strong as a module of $G$.
	On the other hand, if $U$ is not strong as a module of $G$, then there is a module $N$ of $G$ overlapping $U$. Now, $M \setminus N$ is a module of $G$ (\cite[Prop 1.5.1 (ii)]{brandt}), and thus a module of $G_M$. But $M\setminus N$ overlaps $U$, so $U$ is not strong as a module of $G_M$.
\end{proof}
\noindent If $M$ and $N$ are two disjoint modules of $G$, then one of the following holds:
\begin{enumerate}
	\item Either every vertex of $M$ is connected to every vertex of $N$. Then we call $M$ and $N$ connected in $G$ and we write $MN$ for the set of edges between vertices of $M$ and $N$.
	\item Otherwise no vertex of $M$ is connected to any vertex of $N$.
\end{enumerate}
The edges connecting the m.s.s. of a module $M$ are called \emph{external edges} of $M$.
So $M$ is parallel if and only if it has no external edges.
Note that every edge of $G$ is an external edge for exactly one strong module.
We close this section by giving a description of the non-strong modules of $G$:
\begin{lemma}\label{lemma:L2}
	Let $G$ be a graph and let $M$ be a module which is not strong. Then $M$ is the union of some m.s.s. of a parallel or serial strong module.
	On the other hand, any union of m.s.s. of a parallel or serial strong module is a module.
\end{lemma}
\begin{proof}
	Let $N$ be the smallest strong module containing $M$. The m.s.s. of $N$ partition it, so $M$ is a union of some of them. 
	If $N$ is prime, then consider the set $U$ in \ref{thm:moddecomp}. Since $M$ is not strong, it is a union of at least two but not of all m.s.s. of $N$. So $M \cap U$ is a nontrivial submodule of $G_U$, contradicting \ref{thm:moddecomp}.
	Hence $N$ is either serial of parallel.
	
	For the converse, let $M$ be a union of m.s.s. of a serial or parallel strong module $N$. By \ref{prop:untermod}, it suffices to prove that $M$ is a module of $G_{N}$. Let $x,y \in M$ and $m \in N \setminus M$. The edges $xm, ym$ are both external in $N$. But if $N$ is serial, it has all possible external edges and if it is parallel, it has none at all. In both cases, the claim is immediate.
\end{proof}


\subsection{Inversion sets and blocks}
We recall the characterization of those sets that can arise as inversion sets of a permutation.
\begin{proposition}[Proposition 2.2 in \cite{yanagimoto1969partial}, see also \cite{bjoerner1991permutation}] \label{prop:reglabel}
	Let $T \subset \binom{[n]}{2}$ be a subset. The following conditions are equivalent:
	\begin{enumerate}
		\item There exists a permutation $\pi\in\Sn$ with $T=\T$.
		\item For every $1\leq i<j<k \leq n$ it holds that:
		\begin{itemize}
			\item If $ij, jk\in T$, then $ik \in T$.
			\item If $ik \in T$, then either $ij \in T $ or $jk \in T$.
		\end{itemize}
	\end{enumerate} 
\end{proposition}
\noindent If a subset $T \subset \binom{[n]}{2}$ satisfies the conditions of above proposition, say $T = \T$, then so does its complement by $\cc{T} = \T[\wo{n} \pi]$.
We now take a closer look at the modules of the inversion graph of a permutation $\pi \in \Sn$.
Let us call a set $I \subset [n]$ of consecutive integers an \emph{interval}.
\begin{definition}[\cite{brignall2010survey}]
\begin{enumerate}
	\item A \emph{$\pi$-block} is an interval $I\subset [n]$ such that its image $\pi(I)$ is again an interval. 
	\item A $\pi$-block is called \emph{strong} if for every other $\pi$-block $J$ either $I\cap J = \emptyset$, $I\subset J$ or $J \subset I$ holds.
\end{enumerate}
\end{definition}
\noindent The importance of $\pi$-blocks for our purpose stems from the following theorem:
\begin{theorem}\label{thm:stronginterval}
Let $I \subset [n]$ and $\pi\in \Sn$. The following implications hold:
\begin{enumerate}
	\item $I$ is a $\pi$-block $\Longrightarrow$ $I$ is a module of $\G$
	\item $I$ is a strong $\pi$-block $\Longleftrightarrow$ $I$ is a strong module of $\G$
\end{enumerate}
In particular, every strong module of $\G$ is an interval.
\end{theorem}
The first part of this theorem is relatively easy to prove and is mentioned in \cite{young}. 
Its converse fails for trivial reasons: By \ref{lemma:L2}, the non-strong modules of $\G$ are exactly the unions of m.s.s. of parallel or serial strong modules of $\G$. But such a union is not necessarily an interval.
A complete proof of \ref{thm:stronginterval} is included in the appendix.
We call a $\pi$-block parallel, serial or prime if it is a module of this type.

\section{Main results}\label{sec:main}
In this section we prove our main results. Fix a permutation $\pi \in \Sn$.
For $\tau_1,\tau_2 \in \Sn$, we will write $\pi = \tau_1 \decomp \tau_2$ to indicate that the three permutations satisfy \eqref{eq:zerlegung}.
We call $\tau_1 \decomp \tau_2$ an \emph{inv-decomposition} of $\pi$.
If an inv-decomposition of $\pi$ exists, we call $\pi$ \emph{inv-decomposable}.

\subsection{Inversion decomposition}
In this subsection, we describe all possible inv-decompositions of $\pi$.
We start with an elementary observation:
\begin{lemma} \label{lemma:dreieck}
Let $i,j,k \in [n]$ such that $ij, ik \in \G$ and $jk \notin \G$.
Assume that $\pi = \tau_1 \decomp \tau_2$ for $\tau_1,\tau_2 \in \Sn$.
Then $ij, ik$ are both either in $\G[\tau_1]$ or in $\G[\tau_2]$.
\end{lemma}
\begin{proof}
We consider the different relative orders of $i,j$ and $k$ separately, but we may assume $j < k$.
\begin{description}
	\item[$i<j<k$] The edge $ik$ is contained either in $\T[\tau_1]$ or in $\T[\tau_2]$, say in $\T[\tau_1]$. By assumption $jk \notin \T[\tau_1]$, therefore by \ref{prop:reglabel} we have $ij \in \T[\tau_1]$. 
	\item[$j<i<k$] This case is excluded by \ref{prop:reglabel}.
	\item[$j<k<i$] Analogous to the first case.
\end{description}
\end{proof}

Note that there is no assumption on the relative order of $i,j$ and $k$, so this is really a statement about the inversion graph of $\pi$. \ref{lemma:dreieck} gives rise to a partition of the edges of $\G$: Two edges $ij$, $ik \in \G$ with a common endpoint are in the same \emph{edge class} if $jk \notin \G$, and our partition is the transitive closure of this relation. Thus by \ref{lemma:dreieck} two edges in the same class always stay together when we distribute the inversions of $\pi$ on $\tau_1$ and $\tau_2$.
%
In \cite{gallai1967transitiv} edge classes are considered for a different motivation.
In that paper the following description is given\footnote{Note that what we call module is called `geschlossene Menge' (closed set) in \cite{gallai1967transitiv}.}. 

\begin{proposition}[\cite{gallai1967transitiv}] \label{thm:edgeclass}
Let $G=(V,E)$ be a graph with at least two vertices. Then there are two kinds of edge classes:
\begin{enumerate}
	\item For two m.s.s. $M_1, M_2 \subset M$ of a serial module $M$, the set $M_1 M_2$ is an edge class.
	\item The set of external edges of a prime module forms an edge class.
\end{enumerate}
Every edge class is of one of the above types.
\end{proposition}
Edge classes are also considered in \cite[Chapter 5]{golumbic} under the name `colour classes' and in \cite{young} as the connected components of a certain graph $\Gamma_\pi$. Theorem $1$ in the latter reference gives a different characterization of edge classes. 
Now we can state our main result.
We give a description of all ways of partitioning $\T$ into two sets satisfying \eqref{eq:zerlegung}.
\begin{theorem} \label{thm:klass}
	Consider a partition $\T = T_1 \dcup T_2$ of the inversion set of $\pi$ into nonempty subsets $T_1, T_2 \subset \T$. For such a partition, the following conditions are equivalent:
	\begin{enumerate}
		\item There exist permutations $\tau_1, \tau_2 \in \Sn$ such that $T_i = \T[\tau_i]$ for $i=1,2$. In particular, $\pi = \tau_1 \decomp \tau_2$.
		\item For every strong prime module of $\G$, all its external edges are either in $T_1$ or in $T_2$.
		For every strong serial module of $\G$ with $p$ maximal strong submodules $M_1 < \ldots < M_p$ there exists a permutation $\sigma \in S_p$, such that for each pair $1\leq i<j\leq p$ it holds that $M_iM_j \subset T_1$ if and only if $ij \in \T[\sigma]$.
	\end{enumerate}
\end{theorem}

\begin{proof}
\emph{(1) $\Rightarrow$ (2):}
Every edge of $\G$ is an external edge of a module $M$ that is either prime or serial.
If $M$ is a prime module, then its external edges form an edge class, hence they all are in $T_1$ or $T_2$. If $M$ is a serial module with m.s.s. $M_1, \ldots, M_p$, then the sets $M_i M_j$ are edge classes. For every $M_i$, choose a representative $a_i \in M_i$. We construct a permutation $\sigma \in S_p$ as follows: Order the images $\tau_1(a_i), i=1,\ldots,p$ in the natural order. Then $\sigma(i)$ is the position of $\tau_1(a_i)$ in this order. Thus, for $i < j$ we have 
\begin{align*}
	M_i M_j \subset \T[\tau_1] &\Longleftrightarrow \tau_1(a_i) > \tau_1(a_j) \\
	&\Longleftrightarrow \sigma(i) > \sigma(j) \\
	&\Longleftrightarrow ij \in \T[\sigma] \\
\end{align*}

\emph{(2) $\Rightarrow$ (1):}
By symmetry, we only need to show the existence of $\tau_1$. 
For this, we verify conditions of \ref{prop:reglabel}.
This is a condition for every three numbers $1 \leq i < j < k \leq n$, so let us fix them. Note that our hypothesis on $T_1$ and $T_2$ implies that every edge class of $\T$ is contained either in $T_1$ or $T_2$.

Let $M$ be the smallest strong module containing these three numbers. It holds that $i$ and $k$ are in different m.s.s. of $M$, because every strong module containing both would also contain $j$, since it is an interval by \ref{thm:stronginterval}. Now we distinguish two cases: Either, $i$ and $j$ are in the same m.s.s. of $M$, or all three numbers are in different m.s.s..

In the first case, let $i, j \in M_a$ and $k \in M_b$. Then $ik, jk \in M_a M_b$ belong to the same edge class, so either both or neither of them are in $T_1$. This is sufficient to prove that the criterion is satisfied. 

In the second case, the edges $ij, jk, ik$ are all external to $M$.
Hence, if $M$ is prime, either none of them is in $T_1$ or all that are also in $\T$. Since $\T$ is the inversion set of a permutation, the criterion of \ref{prop:reglabel} is clearly satisfied in this case.
If $M$ is serial, then the edges correspond to inversions of $\sigma$: Let $i \in M_a, j \in M_b, k \in M_c$, then $M_a M_b \subset T_1$ if and only if $ab \in \T[\sigma]$ and similarly for the other edges.
Since $\sigma$ is a permutation, the criterion is again satisfied.
\end{proof}

As a corollary, we can count the number of inv-decompositions of $\pi$:
\begin{corollary}\label{cor:anzahl}
Let $m$ be the number of strong prime modules and let $k_i$ be the number of strong serial modules with $i$ maximal strong submodules, $2\leq i\leq n$. The number of inv-decompositions of $\pi$ is
\[ \frac{1}{2}2^m \prod\limits_{i=2}^{n}(i!)^{k_i} - 1 \label{eq:anzahl} \]
In particular, the number of inv-decompositions depends only on the inversion graph $\G$.
\end{corollary}
We exclude the trivial inv-decomposition $\pi = \pi \decomp \id{n}$, therefore the "$-1$" in above formula. The factor $\frac{1}{2}$ is there because we identify $\tau_1 \decomp \tau_2 = \tau_2 \decomp \tau_1$.

\subsection{Multiplicative decompositions}
A notable special case of an inv-decom\-po\-si\-tion is the following: 
\begin{definition}
We call an inv-decomposition $\pi = \tau_1 \decomp \tau_2$ \emph{multiplicative} if $\pi = \tau_1 \tau_2$ or $\pi = \tau_2 \tau_1$ (multiplication as permutations).
\end{definition}
This kind of inv-decomposition is surprisingly common.
In this subsection, we prove the following 
\begin{theorem}\label{thm:mult}
\begin{enumerate}
\item Every inv-decomposable permutation has a multiplicative inv-decomposition.
\item If a permutation $\pi$ has a non-multiplicative inv-decomposition and $\G$ is connected, then $\pi$ has a decreasing subsequence of size $4$.
\end{enumerate}
\end{theorem}
The assumption that $\G$ is connected is needed to avoid a rather trivial case. $\G$ is disconnected if and only if $\pi$ maps a lower interval $[k] \subset [n]$ to itself. So in this case, $\pi$ is the product of a permutation $\pi_1$ on $[k]$ and a permutation $\pi_2$ on $\set{k+1,\ldots,n}$.
If we have multiplicative inv-decompositions $\pi_1 = \tau_{11}\tau_{12}$ and $\pi_2 = \tau_{21}\tau_{22}$, then $\pi = \tau_{11}\tau_{22} \decomp \tau_{12}\tau_{12}$ is in general not multiplicative.
Before we prove \ref{thm:mult}, we prepare two lemmata.
\begin{lemma}\label{lemma:parab}
	If $C \subset [n]$ is the set of vertices of a connected component of $\G$, then $\pi(C) = C$. 
\end{lemma}
\begin{proof}
	Consider $i\in [n]\setminus C$ and $c \in C$. If $i < c$, then $\pi(i) < \pi(c)$ and the same is true for "$>$", thus the claim follows from bijectivity.
\end{proof}
\begin{lemma}\label{prop:multlemma1}
Assume $\pi = \tau_1 \decomp \tau_2$.
If every connected component of $\G[\tau_2]$ is an induced subgraph of $\G$, then $\pi = \tau_1 \tau_2$. 
\end{lemma}
\begin{proof}
We will prove that $\T[\tau_1 \tau_2] = \T$. Let $M_1,\ldots,M_s$ be the vertex sets of the connected components of $\G[\tau_2]$. By \cite[Ex 1.12]{bjoernerbrenti} it holds that 
\[ \T[\tau_1 \tau_2] = \T[\tau_2] \operatorname{\Delta} \tau_2^{-1} \T[\tau_1] \tau_2 \;.\]
Here, $A\operatorname{\Delta} B = A\setminus B \cup B\setminus A$ denotes the symmetric difference.
First, we observe that the two sets are disjoint, thus the symmetric difference is actually a disjoint union.
To see this, note that every edge of $\G[\tau_2]$ has both endpoints in the same $M_k$ for a $1\leq k \leq s$, and every edge of $\G[\tau_1]$ has its endpoints in different sets or in $[n]\setminus \bigcup M_k$. Since by \ref{lemma:parab} it holds that $\tau_2(M_k) = M_k$ for every $k$, this property is preserved under the conjugation with $\tau_2$. Hence, the sets are disjoint.

Next, we prove that every $M_k$ is a $\G$-module.
So fix a $k$, let $M'$ be the smallest strong module of $\G$ containing $M_k$ and let $G_k$ be the subgraph of $\G$ induced by $M_k$.
Because $\pi = \tau_1 \decomp \tau_2$ is a valid decomposition, $G_k$ is a union of edge classes.
Thus, if $M'$ is prime, we conclude that $M_k = M'$ and we are done.
If $M'$ is parallel, then $G_k$ cannot be connected, thus we only need to consider the case that $M'$ is serial.
But in this case, $M_k$ is a union of m.s.s. of $M'$ because of the form of the edge classes, given by \ref{thm:edgeclass}.
By \ref{lemma:L2}, we conclude that $M_k$ is indeed a module of $\G$.
Moreover, it follows that $M_k$ is also a module of $\G[\tau_1]$, because $\G$ and $\G[\tau_1]$ differ only inside the $M_k$.
Finally, consider the set
\begin{align*}
\tau_2^{-1} \T[\tau_1] \tau_2 &= \set{\set{i,j} \in \T[\tau_1]\with i,j \notin M_k \forall k} \\
&\dcup \bigcup_k \set{ \set{\tau_2(i),j} \with \set{i,j} \in \T[\tau_1], i \in M_k , j\notin M_l \forall l}\\
&\dcup \bigcup_k \bigcup_l \set{ \set{\tau_2(i),\tau_2(j)} \with \set{i,j} \in \T[\tau_1], i\in M_k, j \in M_l} \,.
\end{align*}
Because $\tau_2(M_k) = M_k$ and $M_k$ is a module of $\G[\tau_1]$ for all $k$, it holds that
\begin{multline*}
\set{ \set{\tau_2(i),j} \with \set{i,j} \in \T[\tau_1], i \in M_k , j\notin M_l \forall l} = \\
\set{ \set{i,j} \in \T[\tau_1] \with i \in M_k , j\notin M_l \forall l}
\end{multline*}
and
\begin{multline*}
\set{ \set{\tau_2(i),\tau_2(j)} \with \set{i,j} \in \T[\tau_1], i\in M_k, j \in M_l} = \\
 \set{ \set{i,j}\in \T[\tau_1]\with i\in M_k, j \in M_l} \,.
\end{multline*}
Hence $\tau_2^{-1} \T[\tau_1] \tau_2 = \T[\tau_1]$ and the claim follows.
\end{proof}
\begin{proof}[Proof of \ref{thm:mult}]
For the first statement, assume that $\pi$ is inv-decomposable.
Then by \ref{cor:anzahl} there are either at least two non-parallel strong $\pi$-blocks $I_1, I_2 \subset [n]$, or at least one serial strong $\pi$-block $I_3$ with at least three m.s.s..

In the first case, we may assume $I_1 \nsubseteq I_2$. We set $T_2$ to be the set of edges in the induced subgraph of $\G$ on $I_2$.
In the second case, we set $T_2$ to be the set of edges in the induced subgraph of $G$ on the union of the two first m.s.s. of $I$.
In both cases, we set $T_1 = \T \setminus T_2$. By \ref{thm:klass}, this is a valid inv-decomposition, and by \ref{prop:multlemma1} it is multiplicative.

For the second statement, we will prove that $\G$ contains a complete subgraph on $4$ vertices. Let $\pi = \tau_1 \decomp \tau_2$ be a non-multiplicative inv-decomposition. 
Consider a minimal path from $1$ to $n$ in $\G$. If $i$ and $j$ are two vertices in this path that are not adjacent in this path, then they are not adjacent in $\G$, because otherwise we had a shortcut.
Thus by \ref{lemma:dreieck} we conclude that every edge in this path lies in the same edge class.
Hence either $\G[\tau_1]$ or $\G[\tau_2]$ contains a path connecting $1$ with $n$, say $\G[\tau_1]$.
By \ref{lemma:dreieck} this implies that $\G[\tau_1]$ has no isolated vertices.

By our hypothesis and by \ref{prop:multlemma1}, there exists a connected component of $\G[\tau_2]$ that is not an induced subgraph of $\G$. 
Then there exist $1\leq i,j\leq n$ such that $ij\in \G[\tau_1]$ and there is a minimal path $i, i', \ldots, j$ connecting $i$ and $j$ in $\G[\tau_2]$.
By \ref{lemma:dreieck} we have $i'j \in \G$.
We also want to make sure that $i'j \in \G[\tau_2]$.
If this is not the case, then replace $i$ by $i'$.
Then the corresponding statements still hold, but the minimal path is shorter.
Thus, by induction we may assume $i'j \in \G[\tau_2]$.
Since $\G[\tau_1]$ has no isolated vertices, there is a vertex $k$ such that $i'k \in \G[\tau_1]$.
Again by \ref{lemma:dreieck} we conclude that $ik, jk \in \G$.
Thus $\G$ contains the complete subgraph on $i,i',j$ and $k$.
\end{proof}

\subsection{Characterization of inv-decomposability}
We use the results we have proven so far to derive a characterization of inv-decomposability.
Let us recall the definition of the Linear Ordering Polytope.
To every permutation $\pi$ we associate a vector $v_{\pi} \in \mathbb{R}^{n^2}$ by setting
\[
(v_\pi)_{ij} = \begin{cases}
1 &\text{ if } \pi(i) < \pi(j) , \\
0 &\text{ otherwise. }
\end{cases}
\]
The Linear Ordering Polytope is defined to be the convex hull of these vectors.
The inv-decomposability of a permutation $\pi$ can now be characterized as follows.
\begin{theorem} \label{cor:crit}
	For $\pi \in \Sn$ the following statements are equivalent:
	\begin{enumerate}
		\item There exist $\tau_1,\tau_2 \in \Sn\setminus\set{\id{n}}$ such that $\T = \T[\tau_1] \dcup \T[\tau_2]$ and $\pi = \tau_1 \tau_2$, i.e. $\pi$ has a multiplicative inv-decomposition.
		\item There exist $\tau_1,\tau_2 \in \Sn\setminus\set{\id{n}}$ such that $\T = \T[\tau_1] \dcup \T[\tau_2]$, i.e. $\pi$ is inv-decomposable.
		\item $v_{\pi}$ is not a neighbour of the identity in the graph of the linear ordering polytope.
		\item There are at least two edge classes of $\G$.
		\item There are at least two (not necessarily strong) non-trivial non-parallel $\pi$-blocks. (By a non-trivial $\pi$-block, we mean a $\pi$-block that is neither a singleton nor $[n]$)
	\end{enumerate}
\end{theorem}
In \cite{young}, the implications $(2) \Rightarrow (3) \Rightarrow (4) \Rightarrow (5) \Rightarrow (2)$ are proven, 
although the condition $(2)$ is not explicitly mentioned.
As indicated in \ref{sec:intro}, there a gap in the proof. Indeed on page 4 of \cite{young}, in the proof of the implication $(3) \Rightarrow (4)$ the following argument is used.
If $v_{\pi}$ is not a neighbour of $v_{\id{n}}$, then there is a point on the line between the points that can be 
written as convex combination of other vertices, e.g. $\lambda v_{\id{n}} + (1-\lambda) v_{\pi} = \sum \lambda_i v_{\tau_i}$ for $\lambda, \lambda_i \in [0,1]$ and the $\lambda_i$ sum up to $1$.
Considering the support set of the vectors on the left and right-hand side of this equation we obtain an expression $\T = \bigcup \T[\tau_i]$. Note, that in general this union is not disjoint.
In \cite{young}, the existence of this expression, together with the assumption that $\G$ has only one edge class leads to a contradiction, proving $(3) \Rightarrow (4)$.
But $\T[2413] = \set{13,23,24} = \T[2314] \cup \T[1423]$ and $\G[2413]$ has only one edge class, providing a counterexample to above argument.
Since the notation of \cite{young} is different from ours, we provide a full proof of the implications for the convenience of the reader.
\begin{proof}
	\begin{itemize}
		\item[$1 \Leftrightarrow 2:$] \ref{thm:mult}.
		\item[$2 \Rightarrow 3:$] If $\T = \T[\tau_1] \dcup \T[\tau_2]$, then the midpoint of the line connecting $v_{\id{n}}$ and $v_\pi$ is also the midpoint of the line connecting $v_{\tau_1}$ and $v_{\tau_2}$, thus it cannot be an edge.
		\item[$3 \Rightarrow 4:$] If $v_{\pi}$ is not a neighbour of $v_{\id{n}}$, then we can write
		$\lambda v_{\id{n}} + (1-\lambda) v_{\pi} = \sum \lambda_i v_{\tau_i}$ for $\lambda, \lambda_i \in [0,1]$ and $\tau_i \neq \pi$ for every $i$.
		We clear denominators to make the coefficients integral. The important observation is that every non-zero component of the right-hand side has the same value.
		
		Consider $a,b,c \in [n]$ such that $ab, bc \in \T$ and $ac \notin \T$. Then $b$ cannot lie between $a$ and $c$, because of \ref{prop:reglabel}. There remain four possible relative orders of $a,b$ and $c$. We assume $b <a<c$, the other cases follow analogously.
		Every $\tau_i$ with $bc \in \T[\tau_i]$ has also $ba \in \T[\tau_i]$, again by \ref{prop:reglabel}.
		But the number of $\tau_i$ having the inversion $bc$ equals the number of those having $ba$. Hence, every $\tau_i$ has either both or none of the inversions.
		It follows that if $\T[\tau_i]$ contains an inversion, then it already contains the whole edge class of it. Thus if $\G$ has only one edge class, then for every $i$ either $\tau_i = \pi$ or $\tau_i = \id{n}$, which is absurd.
		
		\item[$4 \Rightarrow 5:$] This follows from the description of the edge classes, \ref{thm:edgeclass}.
		\item[$5 \Rightarrow 2:$] Under our hypothesis, the formula in \ref{cor:anzahl} cannot evaluate to zero.
	\end{itemize}
\end{proof}

\subsection{Substitution decomposition}
We give a reformulation of \ref{thm:klass} avoiding notions from graph theory. For this, we employ the concept of \emph{substitution decomposition}, which was introduced in \cite{albert2005simple}, see \cite{brignall2010survey} for a survey.
We start by giving an explicit description of the three types of $\pi$-blocks.
\begin{proposition}\label{prop:intcharak}
	Let $I \subset [n]$ be a $\pi$-block with at least two elements and let $I_1 < \ldots < I_l$ be its maximal strong submodules.
	\begin{enumerate}
		\item $I$ is parallel if and only if $\pi(I_1) < \pi(I_2) < \ldots < \pi(I_l)$.
		\item $I$ is serial   if and only if $\pi(I_1) > \pi(I_2) > \ldots > \pi(I_l)$.
		\item Otherwise $I$ is prime.
	\end{enumerate}
\end{proposition}
\begin{proof}
	This is consequence of \ref{thm:stronginterval}. $I$ is parallel if and only if it has no external edges. This translates to the statement that the relative order of the $I_i$ is preserved.
	Similarly, $I$ is serial if and only if it has all possible external edges. Again, this translates to the statement that the relative order of the $I_i$ is reversed.
\end{proof}
In the remainder of this section, we consider permutations as words $\pi = \pi_1 \pi_2 \ldots \pi_n$. The \emph{size} of a permutation is the number of letters in its word\footnote{This is called `length' in \cite{brignall2010survey} but we reserve that notion for the number of inversions.}.
The special word $\id{n} \defa 1 2 \ldots (n-1) n$ is called an \emph{identity}. If $\pi = \pi_1 \pi_2 \ldots \pi_n$ is a permutation, we call $\rev{\pi} \defa \pi_n \pi_{n-1} \ldots \pi_1$ the \emph{reversal} of $\pi$. The word $\wo{n} \defa \rev{\id{n}} = n (n-1) \ldots 2 1$ is called \emph{reverse identity}.
Two finite sequences $a_1, \ldots, a_q$ and $b_1, \ldots, b_q$ of natural numbers are called \emph{order isomorphic} whenever $a_i < a_j$ if and only if $b_i < b_j$. Given a permutation $\pi \in S_m$ and $m$ further permutations $\sigma_1, \ldots, \sigma_m$ of not necessarily the same size, we define the
\emph{inflation} $\pi[\sigma_1, \ldots, \sigma_m]$ by replacing the value $\pi(i)$ by an interval order isomorphic to $\sigma_i$. For a more detailed treatment of the inflation operation see \cite{wehlau}.
A permutation $\pi$ is called \emph{simple} if there are no other $\pi$-blocks than $[n]$ and the singletons. Note that by \ref{thm:stronginterval} 
 a permutation $\pi$ is simple if and only if its inversion graph $\G$ is prime.

\begin{proposition}
	Every permutation $\pi$ can be uniquely expressed as an iterated inflation, such that every permutation appearing in this expression is either an identity, a reverse identity or a simple permutation, and no identity or reverse identity is inflated by a permutation of the same kind.
\end{proposition}
We call this the substitution decomposition of $\pi$. It is slightly different from the decomposition in \cite{brignall2010survey}. The existence of our decomposition follows from the existence of the decomposition given in that paper, but we consider it to be instructive for our discussion to give a proof nevertheless.
\begin{proof}
	Let $I_1 < I_2 < \ldots < I_l$ be the maximal strong $\pi$-subblocks of $[n]$.
	Define a permutation $\alpha \in S_l$ by requiring $\alpha(i) < \alpha(j) \Leftrightarrow \pi(I_i) < \pi(I_j)$ for $1\leq i,j\leq l$. Moreover, let $\sigma_i$ be the permutation order isomorphic to $\pi(I_i)$ for $1\leq i\leq l$. Then $\pi = \alpha[\sigma_1,\ldots,\sigma_l]$.
	By \ref{thm:moddecomp} the $\pi$-block $[n]$ is either parallel, serial or prime.
	Hence by \ref{prop:intcharak} we conclude that $\alpha$ is either an identity, a reverse identity or simple. By applying this procedure recursively to the $\sigma_i$, we get the claimed decomposition.
	
	The last claim follows also from \ref{thm:moddecomp}, because it implies that no serial module has a maximal strong submodule which is again serial, and the same for parallel modules. This is just the statement that connected components of a graph are connected.
\end{proof}
The proof gives a correspondence between the strong $\pi$-blocks and the permutations appearing in the substitution decomposition. The strong parallel, serial and prime $\pi$-blocks correspond to the identities, reverse identities and simple permutations, respectively.
Now we can reformulate \ref{thm:klass} in terms of inflations:
\begin{algorithm}
	Let $\pi$ be a permutation. Define two new permutations $\tau_1,\tau_2$ in the following way:
	Write down two copies of the substitution decomposition of $\pi$.
	For every simple permutation in it, replace it in one of the copies by an identity.
	For every reverse identity, replace it in one copy by an arbitrary permutation $\sigma$ of the same size and in the other by the reverse $\rev{\sigma}$.
	Then let $\tau_1$ and $\tau_2$ be the permutations defined by these iterated inflations.
\end{algorithm}
\begin{theorem} \label{thm:invdecomp}
	Let $\pi, \tau_1, \tau_2$ be permutations as above and assume that $\tau_1, \tau_2 \neq \id{n}$.
	Then $\pi = \tau_1 \decomp \tau_2$ and every pair $(\tau_1, \tau_2)$ satisfying this condition can be found this way.
\end{theorem}
\begin{proof}
	This is immediate from \ref{thm:klass} using the correspondence described above.
\end{proof}

\section{Further results}\label{sec:further}
In this section, we give some further results.
First, we consider the generalisation of \eqref{eq:zerlegung} to more than two components. It turns out that this case can easily be reduced to the case of two components, as the next proposition shows. 
\begin{proposition} \label{lemma:deg3}
	Let $\pi, \tau_1,\ldots,\tau_l \in \Sn$ be permutations such that $\T = \bigcup{\T[\tau_i]}$ and $\T[\tau_i] \cap \T[\tau_j] = \emptyset$ for $i\neq j$.
	Then for every $1\leq i,j\leq l$ there exists a $\tau_{ij} \in \Sn$ such that $\T[\tau_{ij}] = \T[\tau_i]\dcup\T[\tau_j]$.
\end{proposition}
\begin{proof}
	We show that $T \defa \T[\tau_i] \cup \T[\tau_2]$ satisfies the condition of \ref{prop:reglabel}. Fix $1 \leq a_1 < a_2 < a_3 \leq n$.
	Note that $\cc{T} = \cc{\T[\tau_i]} \cap \cc{\T[\tau_j]}$, so if $a_1 a_2 \notin T$ and $a_2 a_3 \notin T$, then $a_1 a_3 \notin T$.
	On the other hand, if $a_1 a_2, a_2 a_3 \in T$, then $a_1 a_3 \in \T$ and thus $a_1 a_3 \in \T[\tau_k]$ for some $k$. But then $\T[\tau_k]$ contains also $a_1 a_2$ or $a_2 a_3$, therefore $k$ equals $i$ or $j$. It follows that $a_1 a_3 \in T$.
\end{proof}

From \ref{cor:crit} we can derive a simple sufficient (but by no means necessary) condition for a permutation to be inv-decomposable.
\begin{proposition}
	Every permutation $\pi \in \Sn$ with at least $\binom{n}{2} - n +2$ inversions is inv-decomposable
\end{proposition}
\begin{proof}
	Let $\tau \defa \wo{n} \pi$. Then $\T[\tau]$ is the complement of $\T$ and the $\tau$ has at most $n-2$ inversions. So the graph $\G[\tau]$ is disconnected, because it has $n$ vertices, but only $n-2$ edges.
	This means that $[n]$ is a serial module of $\G$.
	If $\pi$ is not inv-decomposable, then $[n]$ can have only two maximal strong submodules, both parallel. But then $\G[\tau]$ would be the disjoint union of two complete graphs. This it not possible with the restriction on the number of edges, as a direct calculation shows.
\end{proof}

\section{Appendix: Blocks and modules}
\newcommand{\naka}[1]{{#1}_{><}}

In this appendix, we prove the following theorem:
\begin{theorem}\label{thm:app}
	Let $I \subset [n]$ and $\pi\in S_n$. The following implications hold:
	\begin{enumerate}
		\item $I$ is a $\pi$-block $\Longrightarrow$ $I$ is a module of $\G$
		\item $I$ is a strong $\pi$-block $\Longleftrightarrow$ $I$ is a strong module of $\G$
	\end{enumerate}
\end{theorem}

For the rest of this section, let $\pi \in S_n$ denote a fixed permutation. For brevity, we write block for $\pi$-block and modules are to be understood as modules of $\G$.
Recall that a block is an interval whose image under $\pi$ is again an interval.
The first statement of \ref{thm:app} is a direct consequence of the following lemma.
\begin{proposition}\label{lemma:L1}
	Let $I \subset [n]$ be an interval. Then $I$ is a module if and only if it is a block.
\end{proposition}
\begin{proof}
	\begin{align*}
		I \textnormal{ module } &\Leftrightarrow \forall i\in [n]\setminus I:  \left[ \exists j\in I: 
		 ij \in\G \Rightarrow \forall j\in I: ij \in\G \right] \\
		 &\Leftrightarrow \forall i\in [n]\setminus I: \left[ \exists j\in I: \pi(i)<\pi(j) \Rightarrow \forall j\in I: \pi(i)<\pi(j) \right] \\
		 &\Leftrightarrow \forall i\in [n]\setminus I: \pi(i) < \pi(I) \textnormal{ or } \pi(i) > \pi(I) \\
		 &\Leftrightarrow I \textnormal{ block } \qedhere\\
	\end{align*}
\end{proof}

We split the proof of the second part of \ref{thm:app} into three lemmata.
For a set $S \subset [n]$ we define $S_< \defa\set{x\in[n]\with x < S}$ and similarly $S_>$. We also define $\naka{S} \defa [n]\setminus (S_< \cup S \cup S_>) = \set{x\in [n] \with \exists a,b\in S: a < x < b, x\notin S}$.
\begin{lemma}\label{lemma:naka}
	Let $M$ be a module. Then $\pi(M_< \cup M_>) = \pi(M)_< \cup \pi(M)_>$ and $\pi(\naka{M}) = \naka{\pi(M)}$.
\end{lemma}
\begin{proof}
	Let $i$ be in $M_<$.
	If $ij\in\G$ for all $j\in M$, then $\pi(i) \in \pi(M)_>$.
	Otherwise $ij\notin\G$ for all $j\in M$ and $\pi(i) \in \pi(M)_<$.
	A similar argument for $i \in M_>$ proves that $\pi(M_< \cup M_>) \subset \pi(M)_< \cup \pi(M)_>$.
	For $i \in \naka{M}$ there exist $j, k\in M$ with $j < i < k$. If $ij\in\G$, then also $ik \in\G$ and therefore $\pi(j) > \pi(i) > \pi(k)$.
	Otherwise $\pi(j) < \pi(i) < \pi(k)$.
	Hence $\pi(\naka{M}) \subset \naka{\pi(M)}$. Equality follows for both inclusions because $\pi$ is bijective.
\end{proof}
\begin{lemma} \label{lemma:L3}
	Every strong module is a strong block.
\end{lemma}
\begin{proof}
Let $M$ be a strong module but not an interval. We write $M \cup \naka{M} = M_1 \cup M_2 \cup \ldots \cup M_l$ where the $M_i$ are the interval components of $M$ and $\naka{M}$ and $M_1 < M_2 < \ldots < M_l$. We proceed by proving the following list of claims:
\begin{enumerate}
	\item $M \cup \naka{M}$ is a module.
	\item $\naka{M}$ is a module.
	\item Either $\pi(M_1) < \pi(M_2) < \ldots < \pi(M_l)$ or $\pi(M_1) > \pi(M_2) > \ldots > \pi(M_l)$.
	\item $M_1 \cup M_2$ is a module.
\end{enumerate}
The last claim is a contradiction to the assumption that $M$ is strong, because $M_1 \subset M$ and $M_2 \cap M = \emptyset$. Hence $M$ must be an interval. By \ref{lemma:L1} we conclude that it is a block. Every other block is also a module, hence the strongness as a block follows from the strongness as a module.
We prove the claims one after the other:
\begin{enumerate}
	\item From \ref{lemma:naka} we know $\pi(\naka{M}) = \naka{\pi(M)}$ and hence $\pi(M \cup \naka{M}) = \pi(M) \cup \naka{\pi(M)}$. Thus this set is a block and the claim follows from \ref{lemma:L1}.
	\item Because $M \cup \naka{M}$ is a module, by \ref{prop:untermod} it suffices to prove that $\naka{M}$ is a module of $M \cup \naka{M}$. Let $i,j \in \naka{M}, k\in M$ and $ik\in\G$. We need to prove $jk\in\G$. Choose $k_1, k_2 \in M$ such that $k_1 < i,j < k_2$. Because $ik \in\G$ and $M$ is a module we know that $k_1 i,i k_2\in\G$. Now we use \ref{prop:reglabel} to conclude:
	\begin{align*} 
	k_1 i,i k_2\in\G &\Rightarrow k_1 k_2\in\G \\
	&\Rightarrow k_1 j,j k_2\in\G \\
	&\Rightarrow j k\in\G
	\end{align*}
	\item It suffices to prove for every $1 < i < l$: Either $\pi(M_{i-1}) < \pi(M_i) < \pi(M_{i+1})$ holds or the corresponding statement with '$>$' holds. If this were wrong, there are $x_k\in M_k, k\in\set{i-1,i,i+1}$ with $\pi(x_{i-1}) > \pi(x_i) < \pi(x_{i+1})$ or $\pi(x_{i-1}) < \pi(x_i) > \pi(x_{i+1})$, say, the first. But then $x_{i-1} x_i \in\G$ and $x_i x_{i+1}\notin \G$. But both edges are in $M \naka{M}$, so this is a contradiction to the previous claim.
	\item Since $M_1 \cup M_2$ is an interval, by \ref{lemma:L1} it suffices to prove that $\pi(M_1 \cup M_2)$ is also an interval. For $x \in [n]\setminus (M_1 \cup M_2)$, it holds that either $x \in M_< \cup M_>$ or $x \in M_3 \cup \ldots \cup M_l$. In the first case we know by \ref{lemma:naka} that $\pi(x) \in \pi(M)_< \cup \pi(M)_> \subset \pi(M_1 \cup M_2)_< \cup \pi(M_1 \cup M_2)_>$. For $x \in M_3 \cup \ldots \cup M_l$ it follows from the previous claim that $\pi(x) \in \pi(M_1 \cup M_2)_< \cup \pi(M_1 \cup M_2)_>$. Therefore, $\pi([n]\setminus (M_1 \cup M_2)) \subset \pi(M_1 \cup M_2)_< \cup \pi(M_1 \cup M_2)_>$. Because $M_1 \cup M_2$ is an interval we can conclude from this that $\naka{\pi(M_1 \cup M_2)} = \emptyset$, thus the claim follows.
\end{enumerate}
\end{proof}
\begin{lemma}
Every strong block is a strong module.
\end{lemma}
\begin{proof}
	Suppose $I \subset [n]$ is a strong block.
	By \ref{lemma:L1} $I$ is a module.
	Thus it remains to prove that it is strong, so assume the contrary.
	By \ref{lemma:L2} it is the union of m.s.s. of a strong module $M'$. Write $M' = M_1 \cup \ldots \cup M_l$, where the $M_i$ are the m.s.s. We have already proven in \ref{lemma:L3} that they are intervals.
	Choose two consecutive ones $M_i, M_{i+1}$ such that $M_i \subset I$ and $M_{i+1} \cap I = \emptyset$. Then $M_i \cup M_{i+1}$ is an interval by construction and a module by \ref{lemma:L2}.
	Therefore it is a block by \ref{lemma:L1}. But this is a contradiction to the hypothesis that $I$ is strong.
\end{proof}

\bibliography{LiteraturModZer}
\bibliographystyle{amsplain}
\end{document}